\documentclass[12pt,reqno]{amsart}
\usepackage{amscd}
\usepackage{amsfonts}
\usepackage{amsmath}
\usepackage{amssymb}
\usepackage{physics}
\usepackage{setspace}	
\usepackage{wrapfig}
\usepackage{amsthm}
\usepackage{fancyhdr}
\usepackage{latexsym}
\usepackage[colorlinks=true, pdfstartview=FitV, linkcolor=blue,
            citecolor=blue, urlcolor=blue]{hyperref}
\usepackage{enumerate}  
\usepackage{mathtools}            
\usepackage{upgreek}
\usepackage{extarrows}

\usepackage{cancel}

\input epsf
\input texdraw
\input txdtools.tex
\input xy
\xyoption{all}

\usepackage{tikz}
\usepackage{tikz-cd}

\usetikzlibrary{patterns}
\usetikzlibrary{calc,matrix, arrows}
\usetikzlibrary{snakes}
\usetikzlibrary{arrows.meta}
\usepackage{subfigure}
\usetikzlibrary{decorations.pathreplacing}
\tikzstyle{level 1}=[level distance=3.5cm, sibling distance=8cm]
\tikzstyle{level 2}=[level distance=3.5cm, sibling distance=3cm]
\tikzstyle{level 3}=[level distance=5cm, sibling distance=1.4cm]

\usepackage{pgfplots}
\pgfplotsset{%
every x tick/.style={black, thick},
every y tick/.style={black, thick},
every tick label/.append style = {font=\footnotesize},
every axis label/.append style = {font=\footnotesize},
compat=1.12
  }
  \usepackage{blindtext}

\usepackage{color}

\definecolor{Red}{rgb}{1.00, 0.00, 0.00}
\definecolor{DarkGreen}{rgb}{0.00, 1.00, 0.00}
\definecolor{Blue}{rgb}{0.00, 0.00, 1.00}
\definecolor{Cyan}{rgb}{0.00, 1.00, 1.00}
\definecolor{Magenta}{rgb}{1.00, 0.00, 1.00}
\definecolor{DeepSkyBlue}{rgb}{0.00, 0.75, 1.00}
\definecolor{DarkGreen}{rgb}{0.00, 0.39, 0.00}
\definecolor{dgreen}{RGB}{0,200,100}
\definecolor{ddgreen}{RGB}{0,170,0}
\definecolor{SpringGreen}{rgb}{0.00, 1.00, 0.50}
\definecolor{DarkOrange}{rgb}{1.00, 0.55, 0.00}
\definecolor{OrangeRed}{rgb}{1.00, 0.27, 0.00}
\definecolor{DeepPink}{rgb}{1.00, 0.08, 0.57}
\definecolor{DarkViolet}{rgb}{0.58, 0.00, 0.82}
\definecolor{SaddleBrown}{rgb}{0.54, 0.27, 0.07}
\definecolor{Black}{rgb}{0.00, 0.00, 0.00}
\definecolor{dark-magenta}{rgb}{.5,0,.5}
\definecolor{myblack}{rgb}{0,0,0}
\definecolor{darkgray}{gray}{0.5}
\definecolor{lightgray}{gray}{0.75}

\newcommand{\R}{\mathbb{R}}

\usepackage{scalerel}
\usepackage{stackengine}
\stackMath
\def\hatgap{-5.5pt}
\def\subdown{-3.3pt}
\newcommand\what[2][]{%
\renewcommand\stackalignment{l}%
\stackon[\hatgap]{#2}{%
\stretchto{%
    \scalerel*[\widthof{$#2$}]{\kern-.6pt\text{\textasciicircum}\kern-1.1pt}%
    {\rule[-0.8\textheight]{1ex}{\textheight}}
}{2ex}
_{\smash{\belowbaseline[\subdown]{\scriptscriptstyle#1}}}%
}}

\newcommand{\rr}{\mathbb{R}}

\def\medno{\medskip \noindent}

\makeatletter
\def\mathcolor#1#{\@mathcolor{#1}}
\def\@mathcolor#1#2#3{%
  \protect\leavevmode
  \begingroup
    \color#1{#2}#3%
  \endgroup
}
\makeatother

\DeclareMathSizes{12}{12}{7}{5}

\def\refer #1\par{\noindent\hangindent=\parindent\hangafter=1 #1\par}

\theoremstyle{plain}  
\newtheorem{theorem}{Theorem}[section]
\newtheorem{proposition}{Proposition}[section]
\newtheorem{lemma}{Lemma}[section]

\usepackage[breakable, theorems, skins]{tcolorbox}
\tcbset{enhanced}

\theoremstyle{definition}

\newtheorem{remark}{Remark}[section]

\newenvironment{Proof}[1][\proofname]
{\proof[\textnormal{\textbf{#1.}}]}{\endproof}
\newcommand{\bp}{\begin{Proof}}
\newcommand{\ep}{\end{Proof}}

\numberwithin{figure}{section}
\numberwithin{equation}{section}

\usepackage[framemethod=TikZ]{mdframed}

\def\ra{\right>}
\def\la{\left<}
\def\lm{\left\|}
\def\rm{\right\|}

\newcommand{\sgn}{\text{sgn}}

\begin{document}
\title[Evolution of the radius of analyticity]{	Evolution of the radius of analyticity for the generalized Benjamin equation }
\author{Renata O. Figueira and Mahendra  Panthee}
\address{Department of Mathematics, University of Campinas\\13083-859 Campinas, SP, Brazil}
\thanks{This work is partially supported by FAPESP, Brazil.}

\maketitle
\onehalfspacing

%
%
%
%
%
%
%
%
%
\begin{abstract}
In this work  we  consider the initial value problem (IVP) for the generalized Benjamin equation
  \begin{equation}\label{Benj-IVP}
\begin{cases}
\partial_t u-l\mathcal{H} \partial_x^2u-\partial_x^3u+u^p\partial_xu
=
0,
\quad x,\; t\in \rr,\; p\geq 1, \\
u(x,0)
=
u_0(x),
\end{cases}
\end{equation}
where $u=u(x,t)$ is a real valued function, $0<l<1$ and $\mathcal{H}$ is the Hilbert transform. This model  was introduced by Benjamin \cite{5} and  describes unidirectional propagation of long waves in a two-fluid system where the lower fluid with greater density is infinitely deep and the interface is subject to capillarity.

We prove that the local solution to the IVP \eqref{Benj-IVP} for given data in the spaces of functions analytic on a strip around the real axis continue to be analytic without shrinking the width of the strip in time. We also  study the evolution in time of the radius of spatial analyticity and show that it can decrease as the time advances. Finally, we present an algebraic lower bound on the possible rate of decrease in time of the uniform radius of spatial analyticity.

\end{abstract}
\vskip 0.3cm

 
\vspace{0.2cm} {\it 2020 AMS Subject Classification:}  35A20, 35Q53, 35B40, 35Q35. 

\vspace{0.5cm}
\section{Introduction}

In this paper, we will present some results concerning the following initial value problem (IVP) for the generalized Benjamin (g-Benjamin) equation 
\begin{equation}\label{gBenj-IVP}
\left\{\begin{array}{l}
\partial_t u-l\mathcal{H} \partial_x^2u-\partial_x^3u+u^p\partial_xu
=
0,
\quad x,\; t\in \rr , \\
u(x,0)
=
u_0(x),
\end{array}\right.
\end{equation}
where $p\in\{1,2,3,\ldots\}$, $0<l<1$ and $\mathcal{H}$ denotes the Hilbert transform defined by
\begin{equation}\label{H-transform}
\mathcal{H}f
=
\text{p.v.} \frac 1\pi \int \frac{f(y)}{y-x} \dd y
=
\mathcal{F}^{-1}_\xi \big(-i\cdot \sgn(\xi)\cdot \widehat{f}(\xi)\big),
\end{equation}
considering the initial data $u_0$ in a class of analytic functions 
that can be extended holomorphically in a symmetric strip 
$
S_\sigma
:=\{x+iy;\;|y|<\sigma\}, \sigma>0,$
of the complex plane around the $x$-axis. 

We are considering the Fourier transform of a function $f$ given by the formula
$$
\widehat{f}(\xi)
=
\mathcal{F}(f)(\xi)
=
\int e^{-ix\xi} f(x) \dd x.
$$
We also use $\mathcal{F}_x$ to denote the partial Fourier transform with respect to the $x$ variable and $\mathcal{F}^{-1}_\xi$ to denote the inverse Fourier transform respect to the variable $\xi$ and similarly for the variable $t$. We will simply use $\mathcal{F}(f)$ or $\widehat{f}$ to denote the Fourier transform  of a function $f$ of any variables if there is no confusion on the number of variables considered.

The equation \eqref{gBenj-IVP} with $p=1$ posed on spatial domain $\mathbb{R}$ was derived by Benjamin
\cite{5} to study the gravity-capillarity surface waves of solitary type in deep water and serves as a generic model for unidirectional propagations of long waves in a two-fluid system where the lower fluid with
greater density is infinitely deep and the interface is subject to capillarity.
The author in \cite{5} also showed that the following quantities 
\begin{equation}\label{cons-1}
M(u):=\frac{1}{2}\int_{\mathbb{R}} u^{2}(x,t)\,\dd x = M(u_0)
\end{equation}
and
\begin{equation}\label{cons-2}
E(u):=\int_{\mathbb{R}} \left[\frac{1}{2}(\partial_{x}u)^{2}(x,t)-
\frac{\alpha}{2}u(x,t)\mathcal{H}\partial_{x}u(x,t)-\frac{2}{(p+1)(p+2)}u^{p+2}(x,t)\right]\,\dd x = E(u_0),
\end{equation}
are conserved by the flow of \eqref{gBenj-IVP}.

The model in \eqref{gBenj-IVP} for $p=1$ is widely known as the Benjamin equation and for $p\geq 2$ as the generalized Benjamin equation.

The well-posedness issues  of the IVP \eqref{gBenj-IVP} for $p=1$ with given data in the classical Sobolev spaces  $H^{s}(\mathbb{R})$ and $H_p^{s}(\mathbb{T})$ have extensively  been studied for many years, see  for example \cite{KOT,20,22,23,Linares Scialom} and references therein. It is worth mentioning the work in \cite{KOT} and \cite{CX} where the authors considered the IVP \eqref{gBenj-IVP} for $p=1$ with data in Sobolev spaces of negative regularity. The authors in \cite{KOT} and \cite{CX} proved that the IVP \eqref{gBenj-IVP} for $p=1$ is locally well-posed in $H^s(\R)$ whenever $s>-\frac34$ and ill-posed in the sense that the flow-map fails to be $C^3$ at the origin when $s<-\frac34$. This result was further improved in  \cite{20} by proving the local well-posedness  for data in $H^{-\frac34}(\R)$. As far as we know, the best known global well-posedness result is obtained in \cite{23} and that holds for data in  $L^{2}(\mathbb{R})$.
In the periodic case, the best local well-posedness result is obtained for data in $H_{p}^{s}(\mathbb{T})$ for $s\geq -\frac{1}{2}$ \cite{22} and the global well-posedness in $L^{2}(\mathbb{T})$ \cite{23}. 

We note that the well-posedness results for the low regularity Sobolev data were obtained using the contraction mapping principle on the Fourier transform restriction norm spaces $X^{s,b}(\rr^2)$ introduced by Bourgain in the seminal paper \cite{13}. More precisely, the Bourgain space $X^{s,b}(\rr^2)$ related to the generalized Benjamin equation is defined as the completion of the Schwartz space with respect to the norm
\begin{equation}\label{Bourgain-norm}
\|u\|_{X^{s,b}}
=
\Bigg( \iint \la \xi\ra^{2s} \la \tau-\phi(\xi)\ra^{2b} |\widehat{u}(\xi,\tau)|^2\dd\xi\dd\tau \Bigg)^{\frac 12},
\end{equation}
where $\la x\ra:= 1+|x|$ and $\phi$ is the phase function related to our problem, that is, 
\begin{equation}\label{phase.func}
\phi(\xi)
=
l|\xi|\xi-\xi^3.
\end{equation}

The Benjamin equation also admits solitary waves solutions. Several works have been devoted to study the existence, stability and asymptotic properties of such solutions, see for instance ~\cite{15,16,5,18} and references therein. Benjamin equation has also been studied in the context of controllability and stabilization, see for example \cite{PV-1, PV-2}. For the generalized Benjamin equation we refer the readers to \cite{Linares Scialom} where the authors not only consider the generalized nonlinearity but also the generalized dispersion. 

We would like to mention that the  IVP \eqref{gBenj-IVP} associated to the generalized Benjamin equation admits  global solution  for $H^1(\R)$ data without any restriction when $1\leq p<4$, however for $p=4$ data should  not be too large and for $p>4$ only for small data. Also the solitary wave solutions are stable for $1\leq p<4$ and unstable for $p>4$, see \cite{16, Ang}. These behaviours of the solutions to the generalized Benjamin equation are similar to that of the gKdV equation.


In the recent time, much efforts have been devoted to find solutions $u(x,t)$ to the IVPs associated to dispersive equations with real-analytic initial data $u_0$ which admit extension as an analytic function to a complex strip $S_{\sigma_0}$, for some $\sigma_0>0$ at least for a short time, i.e., local well-posedness. After getting the local result, a natural question one may ask is whether this property holds globally in time, but with a possibly smaller radius of analyticity $\sigma(t)>0$. In other words, is the solution  $u(x,t)$  with real-analytic initial data $u_0$ analytic on $S_{\sigma(t)}$ for all $t$ and what is the lower-bound of $\sigma(t)$?

Among others, one of the motivation to persue this sort of study comes from numerical analysis \cite{SS-18}. Roughly, $\sigma(t)$ is the distance from
the $x$-axis to the nearest complex singularity of the holomorphic extension of the solution at time $t$. 
If at some time $t$ this singularity  hits the x-axis, then the solution itself suffers a breakdown of regularity. 
This point of view is the basis for the
widely used {\em singularity tracking method}  in numerical analysis, where a spectral
method is used to obtain a numerical estimate of $\sigma(t)$. This estimate can then be used
to predict either the formation of a singularity in finite time or alternatively global regularity.

Analytic Gevrey class introduced by Foias and Temam \cite{FT} is a suitable function space for this purpose. 
An early work in this direction is due to Kato and Masuda \cite{KM}.  They considered a large class of evolution equations and developed a general method to obtain spatial analyticity of the solution. In particular, the class considered in \cite{KM} contains the KdV equation. Further development  in this field can be found in the works of Hayashi \cite{H}, Hayashi and Ozawa \cite{HO},  de Bouard, Hayashi and Kato \cite{BHK}, Kato and Ozawa \cite{KO}, Bona and Gruji\'c \cite{BG-1}, Bona, Gruji\'c and Kalisch \cite{BGK, BGK-2}, Gruji\'c and Kalisch \cite{GK-1, GK-2}, Zhang \cite{Z1, Z2} and references there in. In recent literature, many authors have devoted much effort to get analytic solutions to several evolution equations, see for example \cite{BHG-17, HP-20, HP-18, HPS-17, SS-19, SS-18, SS,  SdS-15, ST-17} and references therein. The local analytic well-posedness and the persistence of the radius of analyticity for global in time solutions were investigated in many other works.  We refer the works in \cite{FHY} and \cite{BFH} for the KdV and modified KdV equation with high dispersion where the authors extended in time the solution and established an algebraic lower bound for the radius of spatial analyticity depending on the dispersion of the equations. Also, a similar result was shown in \cite{FH} where a coupled system of modified KdV equations were considered and the lower bound for the radius of analyticity obtained was $T^{-(2+\varepsilon)}.$
In these works an almost conservation law was developed for each case using multilinear estimates in the Bourgain spaces.

In particular we mention the result on the evolution of radius of analyticity for the KdV equation in \cite{SS} where the authors obtained an algebraic lower bound $ct^{-\frac43+\epsilon}$. Also, algebraic lower bound for the radius of analyticity is found for the gKdV equation, viz $ct^{-(p^2+3p+2)}$ for $p\geq 2$, see \cite{BGK}. Further, for $p=3$, this lower-bound has been improved in \cite{ST-17} by proving it to be $ct^{-2}$. As the well-posedness and stability of solitary waves reults for the g-Benjamin equation match to that of the gKdV equation, a natural question is that whether such is the case in the evolution of radius of analyticity.

The main interest of this work is to answer the question raised above. More precisely, we will  find solutions $u(x,t)$ of the  IVP \eqref{gBenj-IVP} with real-analytic initial data $u_0$ which admit extension as an analytic function to a complex strip $S_{\sigma_0}$, for some $\sigma_0>0$ at least for a short time. After getting the local result, we will examine whether this property holds globally in time what is the lower-bound of $\sigma(t)$? 
In fact, as can be seen in Theorems \ref{global.sol.Benj} and \ref{global.sol.gBenj}
 below, we will obtain an  analytic  global in time solution to \eqref{gBenj-IVP} for $p=1$ and guarantee  algebraic lower bounds for the radius of analyticity for all $p\ge 1$.

 We finish this section recording the organization of this work.   In the next section, we introduce the function spaces and state the main results.  Section \ref{linear-section} provides some linear and preliminary estimates that will be used throughout the work.  
Section \ref{nonlinear-section} is devoted to provide nonlinear estimates. The proofs of the main results stated in Section \ref{sec-2} are supplies in Sections \ref{gwp-section1} and \ref{gwp-section2}. Finally, in Section \ref{concluding remarks}  some concluding remarks are recorded.


\section{Function spaces and the main results}\label{sec-2}

As mentioned in the introduction,  given  $s\in \R$ and $\sigma>0$, the Gevrey spaces $G^{\sigma, s}$ defined by
$$G^{\sigma,s}(\mathbb{R})
:=
\left\{ f\in L^2(\mathbb{R});\;
\|f\|_{G^{\sigma,s}(\mathbb{R})}^2
=
\int \la\xi\ra^{2s}e^{2\sigma |\xi|}|\hat{f}(\xi)|^2 \dd{\xi}
<
\infty\right\},$$
are the adequate spaces to investigate the existence of the analytic solution. Note that by Paley-Wiener theorem, every $f\in G^{\sigma,s}(\mathbb{R})$ has  a holomorphic extension to the strip $S_{\sigma}$, see \cite{SS}. Note that, these spaces enjoy the following embedding
\begin{equation}
\label{Gds.emb}
G^{\sigma, s}(\R)
\subset 
G^{\sigma', s'}(\R),
\end{equation}
for all $ 0<\sigma'<\sigma$ and $ s, s'\in\R$.
 
Since we are working in the analytic case, we should consider the analytic version of Bourgain spaces by adding an exponential term in the norm \eqref{Bourgain-norm} and getting 
\begin{equation}\label{Bourgain-norm-exp}
\|u\|_{X^{\sigma, s,b}}
=
\Bigg( \iint e^{2\sigma|\xi|}\la \xi\ra^{2s} \la \tau-\phi(\xi)\ra^{2b} |\widehat{u}(\xi,\tau)|^2\dd\xi\dd\tau \Bigg)^{\frac 12}.
\end{equation}
Also, for $T\geq 0$,  $X_T^{\sigma,s,b}$ denotes the restricted in time Bourgain space, with the norm given by 
\begin{equation*}
\lm u\rm_{X_T^{\sigma,s,b}}
=
\inf\limits_{v\in X^{\sigma,s,b}}\left\{
\lm v\rm_{X^{\sigma,s,b}};\;v(x,t)
=
u(x,t)
\;\;\mbox{on}\;\; \mathbb{R}\times(-T,T)
\right\}.
\end{equation*}

We introduce the operator $e^{\sigma|D_x|}$ given by 
\begin{equation}
\label{op.A}
\widehat{e^{\sigma|D_x|} u}
=
e^{\sigma |\xi|}\widehat{u},
\end{equation}
to translate the results in the classical Bourgain spaces to the analytic version of them, since we have
\begin{equation*}
\|e^{\sigma |D_x|} u\|_{X^{s,b}}
=
\|u\|_{X^{\sigma, s, b}}.
\end{equation*}
As a consequence we can obtain the following local well-posedness result in the analytic function spaces in line with the one  obtained in \cite{KOT} for initial data in the classical Sobolev spaces. 

\begin{theorem}[Local analytic well-posedness $p=1$]
\label{lwp-thm}
Let $p=1$,  $\sigma >0$ and $s\in( -3/4,0]$.
Then, for all initial data 
$
u_0
\in
G^{\sigma,s}(\mathbb{R})
$
there exist a lifespan $T_0=T_0(\|u_0\|_{G^{\sigma,s}})>0$ and a unique solution $u$ 
of the IVP \eqref{gBenj-IVP} such that
$
u
\in 
C\big([-T_0,T_0];G^{\sigma,s}(\mathbb{R})\big)\cap X^{\sigma,s,b}_{T_0},
$
for some $b> 1/2$.
Moreover, the data-to-solution map is continuous.
Also, there are constants $c_{0}=c_0(s)>0$ and $a=a(s)>1$ with
$
T_0
=
c_{0}\left(1+\lm u_0 \rm_{G^{\sigma,s}}\right)^{-a}
$
and  the solution $u$ satisfies the following estimate
\begin{equation}\label{bound.u}
 \lm  u\rm_{X_{T_0}^{\sigma,s,b}}\leq C\lm u_0 \rm_{G^{\sigma,s}}.
\end{equation}
\end{theorem}

We would like to point out that an analytic well-posedness result for the Benjamin equation was also obtained by Gruji\'c and Kalisch in \cite{GK} for $s> 5/2$.
However, the result we obtained in Theorem \ref{lwp-thm} improves the one obtained  in \cite{GK} and is important here to guarantee a local solution below the conserved index, that is, for negative values of $s$. This improvement in the Sobolev regularity allows us to get the better lower bound for the radius of analyticity. More precisely, we use the  techniques presented by Selberg and Silva in \cite{SS} to obtain a global analytic solution for the Benjamin equation with a greater lower bound for the radius of analyticity, which is described in the next theorem.

\begin{theorem}[Global analytic solution $p=1$]
\label{global.sol.Benj}
Let $p=1$, $\sigma_0>0$, $s\in (-3/4,0]$ and assume $u_0\in G^{\sigma_{0},s}(\mathbb{R})$.
Then the solution $u$ obtained in Theorem \ref{lwp-thm} extends globally in time, and for any $T>0$, we have
$$
u
\in
C\big([0,T], G^{\sigma(T),s}(\mathbb{R})\big),\;
\text{with }
\sigma(T)
=
\min \left\{ \sigma_{0}, \frac{c}{T^{4/3 +\varepsilon}} \right\},
$$
where $\varepsilon>0$ can be taken arbitrarily small and $c>0$ is a constant depending on $\|u_0\|_ {G^{\sigma_{0},s}}$, $s$ and $\varepsilon$.
\end{theorem}

Now, concerning $p\ge 2$, we have a similar result to Theorem \ref{lwp-thm}, but now using the multilinear estimates directly in analytic Bourgain spaces. The proof of these estimates are obtained by following the technique presented in \cite{BGK}, where the authors proved multilinear estimates for the generalized KdV equation.
Actually, the analytic multilinear estimates for $p\ge 2$, that will be presented in Section \ref{nonlinear-section}, are the key ingredient to prove an analytic local well-posedness and to get lower bounds to the radius of analyticity when the solution evolves in time, as stated in Theorem \ref{lwp-thm-p}  and Theorem \ref{global.sol.gBenj}, respectively.

\begin{theorem}[Local analytic well-posedness $p\ge 2$]
\label{lwp-thm-p}
Let $p\ge 2$, $\sigma >0$ and $s> 3/2$. 
Then, for all initial data 
$
u_0
\in
G^{\sigma,s}(\mathbb{R})
$, 
there exist a lifespan $T_0=T_0(\|u_0\|_{G^{\sigma,s}})>0$ and a unique solution $u$ 
of the IVP \eqref{gBenj-IVP} such that
$
u
\in 
C\big([-T_0,T_0];G^{\sigma,s}(\mathbb{R})\big)\cap X^{\sigma,s,b}_{T_0},
$
for some $b> 1/2$.
Moreover, the data-to-solution map is continuous.
\end{theorem}

\begin{theorem}[Global analytic solution $p\ge 2$]
\label{global.sol.gBenj}
Let $p\ge 2$, $T\ge 1$, $\sigma_0>0$, $s> 3/2$ and assume $u_0\in G^{\sigma_{0},s+1}(\rr)$.
If $u$ is a solution to the IVP \eqref{gBenj-IVP} associated to $u_0$ and satisfying $u\in C\big([-4T,4T]; H^{s+1}(\rr)\big)$, then there is $0<\sigma_1<\sigma_0$ such that
$$
u
\in 
C\big([0,T]; G^{\sigma(T), s+1}(\rr)\big),\;
\text{with }
 \sigma(T)
 =
 \min\{\sigma_1, KT^{-(p^2+3p+2)}\}.
 $$
\end{theorem}

%
%
%
%
%
%
%
%
\hrule
\vskip 0.5cm
\section{Linear and Preliminary estimates}\label{linear-section}

In this section we will present some classical results related to the Gevrey spaces $G^{\sigma, s}$ and analytic version of Bourgain spaces $X^{\sigma,s,b}$. We start with the following result whose proof can be found in \cite{BGK}.
\begin{proposition}
\label{Gds-derivatives}
Let $0<\varepsilon <\sigma$ and $n\in\mathbb{N}$. Then there exists a constant $c$ depending on $\varepsilon$ and $n$, such that
\begin{equation}
\sup\limits_{x+iy\in S_{\sigma -\varepsilon}} |\partial_x^n f(x+iy)|
\le
\|f\|_{G^{\sigma, s}},
\end{equation}
where $S_{\sigma-\varepsilon}$ denotes the symmetric strip around the $x$-axis with length of $\sigma -\varepsilon$.
\end{proposition}

\begin{lemma}\label{inc}
Let $b>1/2$, $s\in\mathbb{R}$ and $\sigma>0$. Then, for all $T>0$, the inclusion 
\begin{equation}
X^{\sigma,s,b}(\mathbb{R}^2)\hookrightarrow C\big([-T,T],G^{\sigma,s}(\mathbb{R})\big)
\end{equation}
is continuous, that is,
\begin{equation}
\sup\limits_{t\in[-T,T]} \lm u(\cdot,t)\rm_{G^{\sigma,s}}
\leq
C \lm u\rm_{X^{\sigma,s,b}},
\end{equation}
for some constant $C>0$.
\end{lemma}
%
%
%
By considering the operator $W$ defined by
\begin{equation}\label{mBenj-V-op}
W(t)\varphi
=
\frac{1}{2\pi} \int e^{ix\xi}e^{i\phi(\xi)t}\widehat{\varphi}(\xi)\dd\xi
=
\mathcal{F}_x^{-1}\left( e^{i\phi(\xi)t}\widehat{\varphi}(\xi)\right),
\end{equation}
where $\phi$ denotes the phase function related to our problem (see \eqref{phase.func}),
a formal solution for the IVP \eqref{gBenj-IVP} can be written as 
\begin{equation}\label{mBenj-sol}
u(x,t)
=
W(t)u_0(x)
-\int_0^{t} W(t-t')(u^p\partial_xu)(x,t') \dd t'
.
\end{equation}

Next, we localize in the time variable by using a cut-off function $\psi\in C_0^\infty(-2,2)$ 
with $0\leq \psi\leq 1$, $\psi(t)=1$ on $[-1,1]$ and for $0<T<1$ 
we define $\psi_T(t)=\psi\left(\frac{t}{T}\right)$. 
%
%
%
%
%
%
%
%
%
%
%

We consider the operator $\Phi_T$ given by
\begin{equation}
\label{sol2-H}
\Phi_Tu
:=
\psi_T(t)W(t)u_0(x)
-
\psi_T(t)\int_0^{t} W(t-t')(u^p\partial_xu)(x,t') \dd t'.
\end{equation} 

Now, the objective is to perform a contraction mapping argument on the analytic version of the Bourgain space to get the fixed point of the operator $\Phi_T$. 
The next two lemmas concern in estimating the Bourgain norm for $\Phi_Tu$ whose proof are classical and can be found in \cite{BGK}, \cite{GK-1}, \cite{SS-18} and references therein.
\begin{lemma}
Let $\sigma\ge  0$, $b>1/2$ and $b-1<b'<0$. Then, for all $T>0$ there is a constant c such that
\begin{equation}
\label{psiT.w.u-est}
\|\psi_T(t)W(t)u_0(x)\|_{X^{\sigma,s,b}}
\le
cT^\frac 12 \|u_0\|_{G^{\sigma,s}}
\end{equation}
and
\begin{equation}
\label{psiT.u-est}
\|\psi_T(t)u(x,t)\|_{X^{\sigma,s,b}}
\le
c \|u\|_{X^{\sigma,s,b}}.
\end{equation}
\end{lemma}

\begin{lemma}
Let $\sigma\ge 0$, $b>1/2$ and $-1/2< b' \le 0 < b<1+b'$. 
\begin{enumerate}[(a)]
\item If $0< T \le 1$, then there is a constant c such that 
\begin{equation}
\label{psiT.int.w.u-est.0}
\Bigg\| \psi_T(t)\int\limits_0^t W(t-t')w(x,t')\dd t'\Bigg\|_{X^{\sigma,s,b}}
\le
cT^{1-(b-b')}\|w\|_{X^{\sigma,s,b'}}.
\end{equation}
\item If $T \ge 1$, then there is a constant c such that 
\begin{equation}
\label{psiT.int.w.u-est.1}
\Bigg\| \psi_T(t)\int\limits_0^t W(t-t')w(x,t')\dd t'\Bigg\|_{X^{\sigma,s,b}}
\le
cT\|w\|_{X^{\sigma,s,b'}}.
\end{equation}
\end{enumerate}
\end{lemma}

Following property of the time restricted Bourgain space will also be helpful. For the proof we refer to \cite{SS} and references therein.
\begin{lemma}
\label{est-XT}
Let $\sigma\geq 0$, $s\in\mathbb{R}$, $-1/2<b<1/2$ and $T>0$. 
Then, for any time interval $I\subset [-T,T]$, we have
$$
\lm \chi_{I}u\rm_{X^{\sigma,s,b}}\leq C\lm u\rm_{X_{T}^{\sigma,s,b}},
$$
where $\chi_{I}$ is the characteristic of $I$ and $C$ depends only on $b$.
\end{lemma}

The following estimates will also be useful in our argument. We start with the following lemma whose proof is given in \cite{SS}.
\begin{lemma}
\label{exp-est-lemma}
For $\sigma>0$, $\theta\in [0,1]$, and $\alpha,\beta,\gamma\in\mathbb{R}$, the following estimate holds
\begin{equation}\label{exp-est}
e^{\sigma|\alpha|}e^{\sigma |\beta|}
-
e^{\sigma|\alpha +\beta|}
\leq
\left[
2\sigma \min\left\{ |\alpha|, |\beta|\right\}
\right]^{\theta} 
e^{\sigma|\alpha|}e^{\sigma|\beta|}
\;\;\text{ and }\;\;
\min\left\{ |\alpha|, |\beta|\right\}
\le
\frac{\la \alpha \ra \la\beta\ra}{\la \alpha +\beta \ra}.
\end{equation}
\end{lemma}

Next, we record the $L^2$-version of the bilinear estimate proved in \cite{KOT}.
\begin{proposition}\label{Prop-L2}
For $s\in (-3/4,0]$, there exists $b\in (1/2,1)$ such that for $b-b'\le \min\{ |s|- 1/2, 1/4 -|s|/3\}$ with $b'\in(1/2, b]$ the following estimates holds
\begin{equation}
\label{L2.bilinear.est}
\Bigg\| \frac{\xi}{\la \tau-\phi(\xi)\ra^{1-b}\la\xi\ra^{-s}} 
\iint  
 \frac{\la \xi-\xi_1\ra^{-s}f(\xi-\xi_{1},\tau-\tau_{1})}{\la\tau-\tau_1-\phi(\xi-\xi_1)\ra^{b'}}
 \frac{\la \xi_1\ra^{-s}f(\xi_1,\tau_1)}{\la\tau_1-\phi(\xi_1)\ra^{b'}}\dd{\xi_{1}}\dd{\tau_{1}}
\Bigg\|_{L^{2}_{\xi,\tau}}
\le
C\|f\|^2_{L^2_{\xi,\tau}}.
\end{equation}
\end{proposition}

Now, it remains to estimate the nonlinear part of the map $\Phi_T$, we are going to do this in the next section.
%
%
%
%
%
%
%
%
\vskip 0.3cm
\hrule
\vskip 0.3cm
\section{ Nonlinear estimates }
\label{nonlinear-section}


We devote this section to present the proof of the multilinear estimates, which we split in the cases $p=1$ and $p\ge 2$. We start with the following bilinear estimate in \cite{KOT} that will be used in the case $p=1$.

\begin{proposition}
For $s\in (-3/4,0]$, there exist $b>1/2$ and a constant $c>0$, depending on $s$ and $b$, such that
\begin{equation}
\label{bilinear.est}
\|\partial_x(u_1u_2)\|_{X^{s,b-1}}
\le
c\|u_1\|_{X^{s,b}}\|u_2\|_{X^{s,b}},
\end{equation}
for all $u_1, u_2\in X^{s,b}$.
\end{proposition}

A similar bilinear estimate in the analytic Bourgain space also holds.
\begin{proposition}
\label{nonlinear.Benj}
Let $\sigma>0$ and $s\in (-3/4,0]$. Then there exist $b>1/2$ and a constant $c>0$, depending on $s$ and $b$, such that
\begin{equation}\label{bilinear.est.analytic}
\|\partial_x(u_1 u_2)\|_{X^{\sigma,s,b-1}}
\le
c\|u_1\|_{X^{\sigma, s,b}}\|u_2\|_{X^{\sigma, s,b}},
\end{equation}
for all $u_1, u_2\in X^{\sigma, s,b}$.
\end{proposition}
\begin{proof}
The proof follows just by applying inequality \eqref{bilinear.est} for $e^{\sigma |D_x|} u_i$ instead of $u_i$ with $i=1,2$, where $e^{\sigma|D_x|}$ is the operator defined in \eqref{op.A} and the trivial inequality $e^{\sigma |\xi|}\le e^{\sigma |\xi-\xi_1|} e^{\sigma |\xi_1|}$ (see \cite{BFH-gB} for a more detailed proof).
\end{proof}

Now we move to prove multilinear estimate directly in the analytic Bourgain space that is crucial to deal with case $p\geq 2$.
\begin{proposition}
\label{nonlinear.gBenj} 
Let $p\ge 2$, $\sigma>0$, $s>3/2$, $b> 1/2$ and $b'<- 1/4$. Then there is a constant $c>0$, depending on $s, b$ and $b'$, such that
\begin{equation}
\label{mult.est}
\big\| \partial_x(u_1\ldots u_{p+1})\big\|_{X^{\sigma,s,b'}}
\le
c\prod\limits_{j=1}^{p+1} \|u_j\|_{X^{s,b}}
+
c\sigma^{\frac 12} \prod\limits_{j=1}^{p+1} \|u_j\|_{X^{\sigma,s,b}},
\end{equation}
for all $u_j\in X^{\sigma, s,b}$.
\end{proposition}

 To present the proof of the multilinear estimate \eqref{mult.est}, we introduce the following notations and some results concerning the $L_x^pL_t^q$-norm estimates related to the generalized Benjamin equation.

Let $\rho\in \rr$ and $f(\xi,\tau)$ a given function, we define $F_\rho$ by its Fourier transform as below
\begin{equation*}
\widehat{F}_\rho(\xi,\tau)
= 
\frac{f(\xi,\tau)}{\la \tau-\phi(\xi)\ra^\rho}.
\end{equation*}
For $s\in\rr$ we consider the Sobolev operators $A^s$ and $\Lambda^s$ 
\begin{equation}
\label{Sobolev.op}
\widehat{A^sv}(\xi,\tau)
= 
\la\xi\ra^s \widehat{v}(\xi,\tau)
\;\;\text{and} \;\;
\widehat{\Lambda^sv}(\xi,\tau)
= 
\la\tau\ra^s \widehat{v}(\xi,\tau),
\end{equation}
regarding the space and time variables, respectively.

Also, we consider the following notation to the $L_x^pL_t^q$-norms
\begin{equation*}
\| v\|_{L_x^pL_t^q}
=
\Big( \int \Big|\int |v(x,t)|^q\dd t\Big|^{p/q} \dd x\Big)^{1/p}.
\end{equation*}

Now, we list below some important inequalities to be used in the proof of the multilinear estimates. The idea of  the proof of the following lemmas is similar to the ones presented in \cite{BGK}, however in our case some care should be taken to deal with the structure of the phase function.  To see this difference we will present a detailed proof for the first result.

\begin{lemma}\label{Lema-41}
If $\rho>1/4$, then there is a constant $c$, depending on $\rho$, such that
\begin{equation}\label{L4_xL2_t-est}
\| A^{\frac 12} F_\rho\|_{L_x^4L_t^2}
\le
c\|f\|_{L^2_\xi L^2_\tau}.
\end{equation}
\end{lemma}

\begin{proof}
First, let us prove the following inequality
\begin{equation}
\label{D-L4_xL2_t-est}
\big\| D^\frac12_x F_\rho \big\|_{L^4_x L^2_t}
\le
c\| f\|_{L^2_\xi L^2_\tau},
\end{equation}
where $D^s$ denotes the homogeneous derivative of order $s$, which means that the operator $D$ satisfies $\widehat{D^s_x \varphi}=|\xi|^s\widehat{\varphi}$.
Using Plancherel's identity, we have
\begin{equation}
\label{D-L4_xL2_t-est-part1}
\big\| D^\frac12_x F_\rho \big\|_{L^2_t}
=
\big\| \mathcal{F}^{-1}_x(\mathcal{F}(D^\frac12_x F_\rho)) \big\|_{L^2_\tau}
=
\Bigg( \int \Bigg| \mathcal{F}^{-1}_x\Bigg( \frac{ |\xi|^\frac 12 f(\xi,\tau)}{\la \tau-\phi(\xi)\ra^\rho}\Bigg) \Bigg| ^2\dd\tau\Bigg)^\frac 12.
\end{equation}
By applying Minkowski inequality for integrals, Hausdorff-Young and Holder's inequalities we obtain
\begin{align}
\label{D-L4_xL2_t-est-part2}
\big\| D^\frac12_x F_\rho \big\|_{L^4_x L^2_t}
&\le \nonumber
c
\Bigg( 
\int \Bigg\|  \frac{ |\xi|^\frac 12 f(\xi,\tau)}{\la \tau-\phi(\xi)\ra^\rho} \Bigg\|^2_{L^{\frac 43}_\xi}\dd\tau 
\Bigg)^\frac 12\\
&\le
c
\Bigg[ \int
\Bigg(\int \frac{ |\xi|^2 }{\la \tau-\phi(\xi)\ra^{4 \rho}}\dd\xi\Bigg)^\frac 12
\Bigg(\int |f(\xi,\tau)|^2 \dd\xi\Bigg) 
\dd\tau\Bigg]^\frac 12,
\end{align}
where in the last inequality we used H\"older's inequality in $\xi$ variable.
Now, since
$
\int \frac{ |\xi|^2 }{\la \tau-\phi(\xi)\ra^{4 \rho}}\dd\xi 
<\infty, 
$
for all $\rho>1/4$,
we conclude that inequality \eqref{D-L4_xL2_t-est} holds true.

In order to prove \eqref{L4_xL2_t-est}, for any $a>0$  we introduce  the high and low frequency operators $P^a$ and $P_a$ defined as follows
\begin{equation}
\label{op.P}
P^a\varphi(x)
=
\frac{1}{2\pi}
\int_{|\xi|\ge a}e^{ix\xi}\widehat{\varphi}(\xi)d\xi
\quad\text{and} \quad
P_a\varphi(x)
=
\frac{1}{2\pi}
\int_{|\xi|\le a}e^{ix\xi}\widehat{\varphi}(\xi)d\xi.
\end{equation}
Then, we get
\begin{equation}
\label{L4_xL2_t-est-HL}
\| A^{\frac 12} F_\rho\|_{L_x^4L_t^2}
\le
\| A^{\frac 12} P^aF_\rho\|_{L_x^4L_t^2} +\| A^{\frac 12} P_aF_\rho\|_{L_x^4L_t^2}.
\end{equation}
Fixing $a\ge 1$ and considering the high frequency part, we observe that
\begin{equation*}
\big\| A^{\frac 12} P^aF_\rho \big\|_{L^2_t}
=
\big\| \mathcal{F}^{-1}_x(\mathcal{F}(A^{\frac 12} P^aF_\rho)) \big\|_{L^2_\tau}
=
\Bigg( \int \Bigg| \mathcal{F}^{-1}_x\Bigg(\frac{ \chi_{\{|\xi|\ge a\}}\la\xi\ra^\frac 12 f(\xi,\tau)}{\la \tau-\phi(\xi)\ra^\rho}\Bigg) \Bigg| ^2\dd\tau\Bigg)^\frac 12,
\end{equation*}
which is similar to the term on the right hand side of \eqref{D-L4_xL2_t-est-part1}.
Therefore, by applying the same steps presented above, we get (see \eqref{D-L4_xL2_t-est-part2})
\begin{align*}
\big\|  A^{\frac 12} P^aF_\rho \big\|_{L^4_x L^2_t}
&\le
c
\Bigg[ \int
\Bigg(\int \frac{\chi_{\{|\xi|\ge a\}}\la\xi\ra^2}{\la \tau-\phi(\xi)\ra^{4 \rho}}\dd\xi\Bigg)^\frac 12
\Bigg(\int |f(\xi,\tau)|^2 \dd\xi\Bigg) 
\dd\tau\Bigg]^\frac 12,
\end{align*}
with
$
\int \frac{\chi_{\{|\xi|\ge a\}}\la\xi\ra^2}{\la \tau-\phi(\xi)\ra^{4 \rho}}\dd\xi
\le
4\int \frac{|\xi|^2}{\la \tau-\phi(\xi)\ra^{4 \rho}}\dd\xi<\infty,
$
since $\la\xi\ra\le 2|\xi|$, for all $|\xi|\ge a\ge 1$. 
Thus, we can establish
\begin{equation}
\label{L4_xL2_t-est-H}
\big\|  A^{\frac 12} P^aF_\rho \big\|_{L^4_x L^2_t}
\le
c\|f\|_{L^2_\xi L^2_\tau}.
\end{equation}

In the same way, by applying again the arguments presented above, we have the following bound for the low frequency term
\begin{align*}
\big\|  A^{\frac 12} P_aF_\rho \big\|_{L^4_x L^2_t}
&\le
c
\Bigg[ \int
\Bigg(\int \frac{\chi_{\{|\xi|\le a\}}\la\xi\ra^2}{\la \tau-\phi(\xi)\ra^{4 \rho}}\dd\xi\Bigg)^\frac 12
\Bigg(\int |f(\xi,\tau)|^2 \dd\xi\Bigg) 
\dd\tau\Bigg]^\frac 12,
\end{align*}
with
$
\int \frac{\chi_{\{|\xi|\le a\}}\la\xi\ra^2}{\la \tau-\phi(\xi)\ra^{4 \rho}}\dd\xi
\le
4a^2\int_{-a}^a \frac{1}{\la \tau-\phi(\xi)\ra^{4 \rho}}\dd\xi
\le
8a^3,
$
since $\la\xi\ra \le 2a$, for all $|\xi|\le a$. 
Consequently, we have the same estimate for low frequency, that is,
\begin{equation}
\label{L4_xL2_t-est-L}
\big\|  A^{\frac 12} P_aF_\rho \big\|_{L^4_x L^2_t}
\le
c\|f\|_{L^2_\xi L^2_\tau}.
\end{equation}
The proof of the lemma is finished by putting together \eqref{L4_xL2_t-est-HL} , \eqref{L4_xL2_t-est-H} and \eqref{L4_xL2_t-est-L}.
\end{proof}


Now we state four more estimates whose proofs follow using the similar technique used in  \cite{BGK} with simple modifications as shown in the proof of Lemma \ref{Lema-41} above.

\begin{lemma} \label{Lema-42}
Let $\rho > 1/2$.
\begin{enumerate}[(i)]
\item There is a constant $c$, depending on $\rho$, such that
\begin{equation}\label{Linft_xL2_t-est}
\| A F_\rho\|_{L_x^\infty L_t^2}
\le
c\|f\|_{L^2_\xi L^2_\tau}.
\end{equation}
\item If $s>3\rho$, then there is a constant $c$, depending on $s$ and $\rho$, such that
\begin{equation}\label{L2_xLinft_t-est}
\| A^{-s} F_\rho\|_{L_x^2L_t^\infty}
\le
c\|f\|_{L^2_\xi L^2_\tau}.
\end{equation}
\item If $s>1/4 $, then there is a constant $c$, depending on $s$ and $\rho$, such that
\begin{equation}\label{L4_xLinft_t-est}
\| A^{-s} F_\rho\|_{L_x^4L_t^\infty}
\le
c\|f\|_{L^2_\xi L^2_\tau}.
\end{equation}
\item  If $s>1/2 $, then there is a constant $c$, depending on $s$ and $\rho$, such that
\begin{equation}\label{Linft_xLinft_t-est}
\| A^{-s} F_\rho\|_{L_x^\infty L_t^\infty}
\le
c\|f\|_{L^2_\xi L^2_\tau}.
\end{equation}
\end{enumerate}
\end{lemma}
%

Finally, the estimates presented in lemmas \ref{Lema-41} and \ref{Lema-42} are sufficient to provide a proof for the multilinear estimate \eqref{mult.est} as shown in \cite{BGK} for the gKdV equation.
\vskip 0.3cm

\hrule
\vskip 0.3cm

\section{Local well-posedness - Proof of theorems \ref{lwp-thm} and \ref{lwp-thm-p}}
\label{lwp-section}

Our main goal in this section is to prove that $\Phi_T$ given in \eqref{sol2-H} is a contraction map for some $T>0$, which will imply that $\Phi_T$ has a unique fixed point, thus resulting in the existence of a solution for the IVP \eqref{gBenj-IVP}. For this purpose we need to estimate the norm of $\Phi_T(u)$ in the analytic Bourgain space $X^{\sigma,s,b}$.

\begin{proposition}
\label{Phi_T-bounded-gBenj}
Let $p\ge 1$, $0< T \le 1$, $\sigma >0$ and $s> s_p$, where $s_1= -3/4$ and $s_p=3/2$, for $p\ge 2$.
If $u_0\in G^{\sigma,s}$, then there exist $b>1/2$, $-1/2<b'\le 0$ and  a constant $c=c(\psi,\sigma,s,b,b')$ such that
\begin{equation}
\label{PhiT.bound}
\lm \Phi_T(u)\rm_{X^{\sigma, s,b}}
\leq 
c\lm u_0\rm_{G^{\sigma,s}}
+
cT^{1-(b-b')}\lm u\rm_{X^{\sigma, s,b}}^{p+1},
\end{equation}
for all $u\in X^{\sigma,s,b}$.
\end{proposition}
\begin{proof}
In fact, for $p=1$ we just need to apply inequalities \eqref{psiT.w.u-est}, \eqref{psiT.int.w.u-est.0} and \eqref{bilinear.est.analytic} to obtain \eqref{PhiT.bound}. 
For $p\ge 2$, inequality \eqref{PhiT.bound}  follows using \eqref{psiT.w.u-est}, \eqref{psiT.int.w.u-est.0} and \eqref{mult.est}  as we can see below
\begin{align*}
\lm \Phi_T(u)\rm_{X^{\sigma, s,b}}
&\le
c\| u_0\|_{G^{\sigma, s}} 
+
 cT^{1-(b-b')}\Big(\|u\|_{X^{s,b}}^{p+1}
+
\sigma^{\frac 12}  \|u\|^{p+1}_{X^{\sigma,s,b}}\Big)\\
&\le
c\| u_0\|_{G^{\sigma, s}} 
+
cT^{1-(b-b')} (1+\sigma^\frac 12)  \|u\|^{p+1}_{X^{\sigma,s,b}},
\end{align*}
since
$
\| u\|^2_{X^{s,b}}
\le
\| u\|^2_{X^{\sigma, s,b}},
$
for all $\sigma>0$.
This finishes the proof of \eqref{PhiT.bound} for all $p\ge 1$.
\end{proof}

\begin{remark}
In the proof of Proposition \ref{Phi_T-bounded-gBenj} we can see that the constant $c$ in \eqref{PhiT.bound} does not depend on $\sigma$ when $p=1$, which is not true for $p\ge 2$. This will be important in the construction of the global solution in Section \ref{gwp-section1}.
\end{remark}

Now, we are able to prove that $\Phi_T$ is a contraction map for suitable $0< T\le 1$ on a closed ball in the Bourgain space $X^{\sigma,s,b}$.

\begin{proposition} \label{PF}
Let $p\ge 1$, $\sigma >0$ and $s> s_p$, where $s_1= -3/4$ and $s_p=3/2$ for $p\ge 2$.
For initial data $u_0\in G^{\sigma,s}$, 
there are $b>\frac 12$ and $T=T(\|u_0\|_{G^{\sigma, s}})>0$ such that 
$
\Phi_T
:
B(r)\longrightarrow B(r)
$
is a contraction, where
$
B(r)
=
\left\{u\in X^{\sigma,s,b};\; \lm u\rm_{X^{\sigma,s,b}}\le r\right\}
$
with
$r=2c\|u_0\|_{G^{\sigma,s}},$
and $c=c(\sigma,\psi,b)$ is the constant that appears in Proposition \ref{Phi_T-bounded-gBenj}.
\end{proposition}
\begin{proof}
In fact, from Proposition \ref{Phi_T-bounded-gBenj} it follows that we can choose $b$ and $b'$ such that
$
\lm \Phi_Tu\rm_{X^{\sigma,s,b}}
\le 
\frac{r}{2} + cT^{1-(b-b')}r^{p+1},
$
for all $u\in B(r)$. Choosing 
$
T
=
\min\left\{1,(4cr^p)^{-\frac{1}{1-(b-b')}}\right\}$, 
we obtain 
$
\lm \Phi_Tu\rm_{X^{\sigma,s,b}}
<
r,
$
for all $u\in B(r)$.
Thus, $\Phi_T$ maps $B(r)$ into itself. Also, for all $u$ and $v$ belonging to $B(r)$ we have
\begin{align}
\label{Phi-contraction}
\lm \Phi_Tu-\Phi_Tv\rm_{X^{\sigma,s,b}} 
\le
cT^{1-(b-b')} \big\| \partial_x(u^{p+1}-v^{p+1}) \big\|_{X^{\sigma,s,b'}},
\end{align}
where we used inequality \eqref{psiT.int.w.u-est.0}.
Now, since
$
u^{p+1}-v^{p+1}
=
\sum\limits_{j=0}^{p} (u-v)u^{p-j}v^j
$
it follows from multilinear estimates (\eqref{bilinear.est.analytic} for $p=1$ and \eqref{mult.est} for $p\ge 2$) that
\begin{align*}
\lm \Phi_Tu-\Phi_Tv\rm_{X^{\sigma,s,b}} 
&\le
cT^{1-(b-b')} \Big(\sum\limits_{j=0}^p\|u\|_{X^{\sigma,s,b}}^{p-j}\|v\|^j_{X^{\sigma,s,b}}\Big)\|u-v\|_{X^{\sigma,s,b}}\\
&\le
cT^{1-(b-b')}r^p\|u-v\|_{X^{\sigma,s,b}},
\end{align*}
for all $u,v\in B(r)$. Thanks to the choice of $T$, we have
$
cT^{1-(b-b')}r^p
<1,
$
showing that $\Phi_T$ is a contraction,
which finishes the proof.
\end{proof}

Finally, we present a unified proof for the local well-posedness results for $p\geq 1$.
\begin{proof}[Proof of Theorems \ref{lwp-thm} and \ref{lwp-thm-p}]

 By Proposition \ref{PF}, we see that for initial data $u_0\in G^{\sigma,s}(\rr)$, 
there is a $0<T_0\le1$ such that $\Phi_{T_0}$ is a contraction on a small ball centered at the origin in $X^{\sigma,s,b}$. 
Hence $\Phi_{T_0}$ has a unique fixed point $u$ in a neighborhood of $0$ with respect to the norm $\|\cdot\|_{X^{\sigma,s,b}}$. 
Since $\psi_{T_0}(t)=1$, for all $|t|\le T_0$, it follows that $u$ solves the IVP \eqref{gBenj-IVP} on $\rr\times [-T_0,T_0]$. 
Finally, thanks to Lemma \ref{inc}, we have proved the existence of a solution to the IVP \eqref{gBenj-IVP}  which belongs to the space $C\big([-T_0,T_0],G^{\sigma,s}(\rr)\big)$ with 
$
\|u\|_{X_{T_0}^{\sigma,s,b}}
\le
r=
2c\|u_0\|_{G^{\sigma,s}},
$
and this guarantees the bound \eqref{bound.u}.

From the fixed point argument used above, 
	we have uniqueness of the solution of
	$\Phi_{T_0}u=u$ in the set $B(r)$. 
	By using the arguments due to 
	Bekiranov, Ogawa and Ponce presented in \cite{BOP} ,
	we obtain the uniqueness in the whole space $X_{T_0}^{\sigma,s,b}$.

Finally, the continuity of the map $\Phi_{T_0}$ also follows from the inequalities obtained in Propositions \ref{Phi_T-bounded-gBenj} and \ref{PF}. 
\end{proof}
%
%
%
%
%
%
%
%
\vskip 0.3cm
\hrule
\vskip 0.3cm
\section{ Global solution for $p=1$ - Proof of Theorem \ref{global.sol.Benj} }
\label{gwp-section1}

In this section we will construct an analytic global in time solution to the IVP \eqref{gBenj-IVP}, when $p=1$. Using the bilinear estimate \eqref{bilinear.est.analytic}, we are able to prove an almost conservation law given below, whose proof uses techniques presented in \cite{SS}.

\begin{proposition}\label{Conser.Law}
Let $p=1$. Given $\theta \in (0,3/4)$ there exist $b\in(1/2,1)$ and $C>0$, such that for any $\sigma,T>0$ 
and any solution $u\in X^{\sigma,0,b}_{T}$ to the IVP \eqref{gBenj-IVP}
on time interval $[0,T]$, we have the estimate
\begin{equation}
\label{alm.conser}
\sup\limits_{t\in [0,T]} \lm u(t)\rm^{2}_{G^{\sigma,0}}
\le
\lm u(0)\rm^{2}_{G^{\sigma,0}}+C\sigma^{\theta}\lm u\rm^{3}_{X_{T}^{\sigma,0,b}}.
\end{equation}
\end{proposition}

\begin{proof}
Let $u_0\in G^{\sigma, 0}(\rr)$ and $u$ be the solution in $C\big([0,T], G^{\sigma, 0}(\rr)\big)$ guaranteed by the local well-posedness result.

We observe that
\begin{equation}
\label{G0-u-norm}
\frac{d}{dt}\|u(t)\|_{G^{\sigma, 0}}^2
=
2\int e^{\sigma|D_x|} u \cdot \partial_t(e^{\sigma|D_x|}u)\, \dd x,
\end{equation}
since we are considering $u$ a real-valued function.

Applying  the operator $e^{\sigma|D_x|}$ to the Benjamin equation, we get
\begin{equation}
\label{A-Benj}
\partial_t(e^{\sigma|D_x|} u)
=
l\mathcal{H}\partial_x^2(e^{\sigma|D_x|} u)+\partial_x^3(e^{\sigma|D_x|} u) -\partial_x(e^{\sigma|D_x|} (u^2)),
\end{equation}
since $e^{\sigma|D_x|}$ is commutative with $\mathcal{H}$ and differential operators.

Using integration by parts and the fact that $e^{\sigma |D_x|}u$ and its derivatives tend to zero when $|x|$ goes to infinity,
it is easy to conclude that the integrals
$$
\int e^{\sigma |D_x|} u\cdot \partial_x^3(e^{\sigma |D_x|} u)\, \dd x,\;\;\;
\int e^{\sigma |D_x|} u\cdot \partial_x((e^{\sigma |D_x|} u)^2)\, \dd x
\;\;\;\mathrm{and}
\;\;\;\int e^{\sigma |D_x|} u\cdot (l\mathcal{H}\partial_x^2(e^{\sigma |D_x|} u)) \,\dd x
$$
are equal to zero.
Thus, from \eqref{G0-u-norm} and \eqref{A-Benj} we obtain 
\begin{align}
\label{G0-u-norm.4}
\frac 12\frac{d}{dt}\|u(t)\|_{G^{\sigma, 0}}^2
=
\int e^{\sigma|D_x|} u\cdot F\dd x,
\end{align}
where $F$ is given by
\begin{equation}
\label{F.operator}
F
=
\partial_x((e^{\sigma |D_x|} u)^2)-\partial_x(e^{\sigma |D_x|} (u^2)).
\end{equation}
If the integral on the right side of \eqref{G0-u-norm.4} was equal to zero, then we would have that the norm $\|u(t)\|_{G^{\sigma, 0}}$  is a conserved quantity, however this is not the case. The next steps are devoted to show an outline of the estimate for the integral in \eqref{G0-u-norm.4} and then obtain an almost conservation law, the reader can found a more detailed proof in \cite{SS}.

Now, integrating in time the last equation \eqref{G0-u-norm.4} from $0$ to $t'\in[0,T]$, we obtain
\begin{equation}
\label{almost-cons-1}
\lm u(t')\rm_{G^{\sigma,0}}^{2}
\leq
\lm u(0)\rm_{G^{\sigma,0}}^{2}
+
2 \left|\iint \chi_{[0,t']}(t) e^{\sigma |D_x|}u \cdot F\dd{x}\dd{t} \right|,
\end{equation}
for all $t'\in[0,T]$.
It follows from Plancherel's theorem, Cauchy-Schwarz inequality and Lemma \ref{est-XT} that
\begin{equation}
\label{almost-cons-2}
\lm u(t')\rm_{G^{\sigma,0}}^{2}
\le
\lm u(0)\rm_{G^{\sigma,0}}^{2}
+
C\lm u \rm_{X_{T}^{\sigma,0,1-b}}\lm F\rm_{X_{T}^{0,b-1}},
\end{equation}                      
for all $t'\in [0,T]$. 

Next, we claim that there exist $b\in (\frac 12, 1)$ and $C>0$,
such that for all $\sigma>0$, $\theta\in (0,3/4)$ and $u\in X^{\sigma,0,b}$, we have
\begin{equation}
\label{est-F}
\lm F\rm_{X^{0,b-1}}
\leq
C \sigma^{\theta}\lm u\rm_{X^{\sigma,0,b}}^{2}.
\end{equation}
We observe that
\begin{align*}
|\widehat{F}(\xi,\tau)|
=
C|\xi|\left|
\iint \left(e^{\sigma|\xi-\xi_{1}|}e^{\sigma|\xi_1|}-e^{\sigma|\xi|}\right)\widehat{u}(\xi-\xi_{1},\tau-\tau_{1})\widehat{u}(\xi_1,\tau_1)\dd{\xi_{1}}\dd{\tau_{1}}
\right|.
\end{align*}

Using \eqref{exp-est} with $\alpha=\xi-\xi_{1}$ and $\beta=\xi_{1}$, we obtain
\begin{align}
\label{F-transform}
|\widehat{F}(\xi,\tau)|
\le
C\sigma^\theta|\xi|\left|
\iint  \frac{\la \xi-\xi_1\ra^\theta \la\xi_1\ra^\theta}{\la \xi\ra^\theta}
e^{\sigma|\xi-\xi_1|} \widehat{u}(\xi-\xi_{1},\tau-\tau_{1})e^{\sigma|\xi_1|}\widehat{u}(\xi_1,\tau_1)\dd{\xi_{1}}\dd{\tau_{1}}
\right|.
\end{align}
Therefore, it follows from \eqref{F-transform}  that
\begin{align}
\label{F-Xnorm}
\lm F\rm_{X_T^{0,b-1}}
&\le
C\sigma^\theta
\Bigg\| \frac{\xi}{\la \tau-\phi(\xi)\ra^{1-b}} 
\iint  \frac{\la \xi-\xi_1\ra^\theta \la\xi_1\ra^\theta f(\xi-\xi_{1},\tau-\tau_{1})f(\xi_1,\tau_1)
}{
\la \xi\ra^\theta\la\tau-\tau_1-\phi(\xi-\xi_1)\ra^{b}\la\tau_1-\phi(\xi_1)\ra^{b}}
\dd{\xi_{1}}\dd{\tau_{1}}
\Bigg\|_{L^{2}_{\xi,\tau}},
\end{align}
where
$
f(\xi,\tau)
=
\la\tau-\phi(\xi)\ra^{b} \widehat{e^{\sigma |D_x|}u}(\tau,\xi).
$
At this point, we  use the $L^2$ version of the bilinear estimates \eqref{bilinear.est} from Proposition \ref{Prop-L2}. More precisely, applying \eqref{L2.bilinear.est} with $s=-\theta \in (-3/4,0)$ and $b'=b$ to \eqref{F-Xnorm}, we obtain from  \eqref{almost-cons-2}  that
\begin{equation*}
\sup\limits_{t\in [0,T]}\lm u(t')\rm_{G^{\sigma,0}}^{2}
\le
\lm u(0)\rm_{G^{\sigma,0}}^{2}
+
C\sigma^\theta\lm u \rm^3_{X_{T}^{\sigma,0,1-b}},
\end{equation*}                   
which finishes the proof of Proposition \ref{Conser.Law}.
\end{proof}

\begin{remark}
An important diference between the cases $p=1$ and $p\ge 2$ lies when we used \eqref{L2.bilinear.est} with $s=-\theta \in (-3/4,0)$ in the proof of Proposition \ref{Conser.Law}, where $\theta$ comes from the technical Lemma \ref{exp-est-lemma}. The fact that we have bilinear estimate below the conserved index $s=0$ is crucial to apply the arguments for the case $p=1$ presented here. For $p \ge 2$ we take another direction, as we will see in the next section.
\end{remark}

Now we are in position to supply a proof for the first global well-posedness result and the lower bound for the radius of analyticity.


\begin{proof}[Proof of Theorem \ref{global.sol.Benj}]
	Let $s \in (-3/4, 0]$, $\sigma_{0}>0$ and $u_0 \in G^{\sigma_{0},s}$.
Observe that, at this point, we just guarantee that there is a local in time solution $u$ belonging to $C\big([0,T_0], G^{\sigma_0,s}\big)$.

	Given $T>0$, let us prove that the solution $u$ belongs to $C\big([0,T], G^{\sigma(T),s}\big)$, where
$
\sigma(T)
=
\min \left\{ \sigma_{0}, \dfrac{c}{T^{4/3 +\varepsilon}}\right\}.
$
The local well-posedness result guarantees the existence of a maximal lifespan
$
T^{*}
=
T^{*}(\|u_0\|_{G^{\sigma_0,s}})
\in
(0,\infty],
$
which means
$
u
\in 
C\big([0,T^{*}), G^{\sigma_{0},s}\big).
$ 
If $T^{*}=\infty$, then we are in the best situation, since the uniform radius of analyticity $\sigma_0$ remains the same for all time in this case. Assuming $T^{\ast}<\infty$, we just need to prove that
$u$ belongs to $C\big([0,T], G^{\sigma(T),s}\big)$, for all $T\ge T^{*}$.

	First, we suppose $s=0$. Also, we fix $\theta\in (0, 3/4)$ and $b=b(\theta)\in(1/2,1)$ as in Proposition \ref{Conser.Law}.
	The proof will be given by applying Theorem \ref{lwp-thm} as many times as necessary to reach time $T$. 
	For this purpose, we use archimedean property to choose a natural number $n_{0}$ such that 
\begin{equation}
\label{delta.choice}
n_{0}\rho
\leq
T
<
(n_{0}+1)\rho,
\;\;\text{where}\;\;
\rho
=
\frac{c_{0}}{\left(1+2\|u_0\|_{G^{\sigma_0,0}}\right)^{a}}
<
T_0,
\end{equation}
with $c_{0}>0$, $a>1$ as in Theorem \ref{lwp-thm}.

\medno
{\bf Claim.}
We claim that if $\sigma$ is chosen in a such way that
\begin{equation}\label{delta.cond.2}
0
<
\sigma
\le
\sigma_{0},
\;\;\text{and}\;\;
n_0C\sigma^{\theta}2^{\frac 32}\|u(0)\|_{G^{\sigma_{0},0}}
\leq
1,
\end{equation}
then solution $u$ can be extended in time satisfying 
\begin{equation}\label{aim}
u\in C\big([0,(n_0+1)\rho]; G^{\sigma,0}(\mathbb{R})\big),
\end{equation}
and
\begin{equation} \label{aim1}
\sup\limits_{t\in [0,(n_0+1)\rho]} \lm u(t)\rm^{2}_{G^{\sigma,0}}
\le
\lm u(0)\rm^2_{G^{\sigma,0}}
+
(n_0+1)C2^{\frac 32}\sigma^{\theta}\lm u(0)\rm^3_{G^{\sigma,0}}.
\end{equation}

In fact, the proof of this claim follows by using induction argument, Theorem \ref{lwp-thm} and the almost conservatio law \eqref{alm.conser}. 

Thus, if $\sigma$ satisfies \eqref{delta.cond.2}, then solution $u$ belongs to $C\big([0,(n_0+1)\rho]; G^{\sigma,0}(\mathbb{R})\big)$.
In particular, under the condition \eqref{delta.cond.2}, we extended the solution $u$ to the space $C\big([0,T]; G^{\sigma,0}(\mathbb{R})\big)$, since $n_0$ was taken satisfying \eqref{delta.choice}.

Therefore, using that $\frac{\rho}{T}\le \frac{1}{n_0}$ (see \eqref{delta.choice}), in order to fulfill the condition \eqref{delta.cond.2} it is sufficient to choose $\sigma$ given by
\begin{equation}
\sigma
=
\left[\frac{\rho}{TC2^{\frac 32}\|u_0\|_{G^{\sigma_0,0}}} \right]^{\frac 1\theta}
=
\left[\frac{c_0}{\left[1+2\|u_0\|_{G^{\sigma_0,0}}\right]^{a}TC2^{\frac 32}\|u_0\|_{G^{\sigma_0,0}}}\right]^{\frac 1\theta}
=:
cT^{-\frac 1\theta},
\end{equation}
where $c$ depends on $a$, $c_{0}$, $\theta$ and $\|u_0\|_{G^{\sigma_0, 0}}$. 
This finishes the proof of Theorem \ref{global.sol.Benj} when $s=0$.

For general $s$, we use the embedding \eqref{Gds.emb} to get
$
u_0
\in 
G^{\sigma_{0},s}(\mathbb{R})
\subset
G^{\sigma_{0}/2,0}(\mathbb{R}).
$
The case $s=0$ already being proved, we know that there is a $T^*\!>0$ such that
$
u\in C\big([0,T^*), G^{\sigma_{0}/2,0}\big)
$
and
$
u\in C\big([0,T], G^{2c'T^{-\frac 1\theta},0}\big),
$
for all $T\ge T^*$,
where $c'>0$ depends on  $a$, $c_{0}$, $\theta$ and $\|u_0\|_{G^{\sigma_0, 0}}$.	

	Applying again the embedding \eqref{Gds.emb}, we now conclude that $u$ belongs to 
$
C\big([0,T^*), G^{\sigma_{0}/4,s}\big)
$
and
$$
u\in C\big([0,T], G^{c'T^{-\frac 1\theta},s}\big),
\text{ for all } T\geq T^*,
$$	
and these together complete the proof of Theorem \ref{global.sol.Benj}.
\end{proof}

%
%
%
%
%
%
%
%
\vskip 0.3cm
\hrule
\vskip 0.3cm
\section{ Lower bounds for the radius of analyticity ($p\ge 2$) - Proof of Theorem \ref{global.sol.gBenj} } \label{gwp-section2}

In this section, we will prove the algebraic lower bound for the radius of analyticity for the solution of the generalized Benjamin equation. We start this section with the following lemma.
\begin{lemma}
\label{psiT.u.Bourgain-norm}
Let $s> -\frac 12$, $b\in [-1,1]$, $T\ge 1$, $\sigma >0$ and $u$ a solution of the IVP \eqref{gBenj-IVP} for $(x,t)\in \rr\times [-2T,2T]$. Then, 
\begin{equation}
\label{psiT.u.Bourgain-norm-sobolev}
\|\psi_Tu\|_{X^{s,b}}
\le
cT^\frac12 \Big(1+\sup\limits_{t\in[-2T,2T]} \| u(\cdot, t)\|_{H^{s+1}}\Big)^{p+1}
\end{equation}
and
\begin{equation}
\label{psiT.u.Bourgain-norm-analytic}
\|\psi_Tu\|_{X^{\sigma, s,b}}
\le
cT^\frac12 \Big(1+\sup\limits_{t\in[-2T,2T]} \| u(\cdot, t)\|_{G^{\sigma,s+1}}\Big)^{p+1}.
\end{equation}
\end{lemma}
\begin{proof}
We are going to present just the proof for  \eqref{psiT.u.Bourgain-norm-sobolev}, since the proof of \eqref{psiT.u.Bourgain-norm-analytic} follows analogously. 

We start by observing that
$
\|\psi_Tu\|_{X^{s,b}}^2
=\int \la\xi\ra^{2s} \big\| \Lambda^{b}\varphi_\xi\big\|^2_{L^2_t} \dd\xi,
$
where we used Plancherel's identity considering $\varphi_\xi(t)=\mathcal{F}^{-1}_t\big[\widehat{\psi_T u}(\xi, \lambda +\phi(\xi))\big](t)$ and $\Lambda^b$ is the operator given in \eqref{Sobolev.op}.

Since $b\le 1$, we get
$
\|\Lambda^b \varphi_\xi\|^2_{L^2_t}
\le
\|\Lambda \varphi_\xi\|^2_{L^2_t}
\le
2\big(\|\varphi_\xi\|_{L^2_t}^2 +\|\varphi'_\xi\|^2_{L^2_t}\big), 
$
which implies
\begin{align}
\label{psiT.u.Bourgain-norm}
\|\psi_Tu\|_{X^{ s,b}}^2
&\le \nonumber
c\Big(
\int \la\xi\ra^{2s} \big\|\varphi_\xi\big\|^2_{L^2_t} \dd\xi
+
\int \la\xi\ra^{2s} \big\|\varphi'_\xi\big\|^2_{L^2_t} \dd\xi 
\Big)
\\
&\le
c
\Big(
\iint \la\xi\ra^{2s}\big|e^{-i\phi(\xi)t}\psi_T(t)\mathcal{F}_x(u)(\xi, t)\big|^2\dd t \dd\xi
+
\iint \la\xi\ra^{2s}\big|\partial_t(e^{-i\phi(\xi)t}\psi_T(t)\mathcal{F}_x(u)(\xi, t))\big|^2\dd t \dd\xi
\Big).
\end{align}
%
Recalling that $\psi_T(t)=\psi\big(t/T\big)$ and $u$ is a solution to the generalized Benjamin equation, we have
\begin{align*}
\partial_t(e^{-i\phi(\xi)t}\psi_T(t)\mathcal{F}_x(u)(\xi, t))
=
&-
i\phi(\xi)e^{-i\phi(\xi)t}\psi_T(t)\mathcal{F}_x(u)(\xi, t)
+
T^{-1}e^{-i\phi(\xi)t}\psi'_T(t)\mathcal{F}_x(u)(\xi, t)\\
&+
e^{-i\phi(\xi)t}\psi_T(t)(i\phi(\xi))\mathcal{F}_x(u)(\xi, t)
-
e^{-i\phi(\xi)t}\psi_T(t)\Big(\frac{i\xi}{p+1}\Big)\mathcal{F}_x(u^{p+1})(\xi, t).
\end{align*}
Now, returning to \eqref{psiT.u.Bourgain-norm}, we get
\begin{align*}
\|\psi_Tu\|_{X^{ s,b}}^2
&\le \nonumber
c\iint \la\xi\ra^{2s}\big|e^{-i\phi(\xi)t}\psi_T(t)\mathcal{F}_x(u)(\xi, t)\big|^2\dd t \dd\xi
+
\frac{c}{T}\iint \la\xi\ra^{2s}\big|e^{-i\phi(\xi)t}\psi'_T(t)\mathcal{F}_x(u)(\xi, t))\big|^2\dd t \dd\xi\\
&\quad\quad
+
\frac{c}{p+1} \iint \la\xi\ra^{2s}\big|e^{-i\phi(\xi)t}\psi_T(t)(i\xi)\mathcal{F}_x(u^{p+1})(\xi, t)\big|^2\dd t \dd\xi.
\end{align*}
Using that $\psi_T$ is a smooth function supported on $[-2T,2T]$, inequality \eqref{psiT.int.w.u-est.0}, $T\ge 1$ and $|e^{-i\phi(\xi)t}|=1$, we obtain
\begin{align*}
\|\psi_Tu\|_{X^{ s,b}}^2
&\le
cT \sup\limits_{t\in [-2T,2T]} \int \la\xi\ra^{2s}\big|\mathcal{F}_x(u)(\xi, t)\big|^2 \dd\xi
+
cT \sup\limits_{t\in [-2T,2T]} \int \la\xi\ra^{2(s+1)} \big|\mathcal{F}_x(u^{p+1})(\xi, t)\big|^2 \dd\xi\\
&\le
cT\Big(\sup\limits_{t\in [-2T,2T]}\|u(\cdot,t)\|^2_{H^s} + \sup\limits_{t\in [-2T,2T]}\|u(\cdot,t)\|^{2(p+1)}_{H^{s+1}}\Big)\\
&\le
cT\Big(1 + \sup\limits_{t\in [-2T,2T]}\|u(\cdot,t)\|_{H^{s+1}}\Big)^{2(p+1)},
\end{align*}
since $s+1\ge \frac 12$, which finishes the proof of  \eqref{psiT.u.Bourgain-norm-sobolev}. 
\end{proof}

In what follows we consider a sequence $\{u_n\}$ that approximates the solution of the generalized Benjamin equation \eqref{gBenj-IVP} defined as follows.
For each $n\in\mathbb{N}$ and $S>0$, we consider the IVP
\begin{equation}
\label{gBenj-IVP-n}
\left\{
\begin{array}{l}
\partial_t u_n-l\mathcal{H} \partial_x^2u_n-\partial_x^3u_n = -\frac{1}{p+1}\partial_x\left[ (\eta_n \ast \psi_S u_n)^{p+1}  \right]\\
u_n(x,0) = u_0(0),
\end{array}
\right.
\end{equation}
where $\{\eta_n\}$ satisfies
\begin{equation}
\label{eta-sequence}
\widehat{\eta_n}(\xi)
=
\left\{
\begin{array}{ll}
0, & |\xi|\ge 2n,\\
1, & |\xi|\le n,
\end{array}
\right.
\end{equation}
with $\hat{\eta}_n$ a monotone sequence on $(-2n,-n)$ and $(n, 2n)$.
It is easy to check  the following properties of the sequence $\{u_n\}$ recorded in the next lemma (for more details see \cite{BGK}).
\begin{lemma}
\label{psi.u.bourgain-n}
Let $\sigma\ge 0$, $r\ge 0$ and $u_0\in G^{\sigma,r}(\rr)$, and let $u$ be a solution of the IVP \eqref{gBenj-IVP} with initial data $u_0$ belonging to $C\big([-2S,2S]; G^{\sigma,r}(\rr)\big)$ for some $S>0$. 
For $n\in\mathbb{N}$, let $u_n$ be the solution of the IVP \eqref{gBenj-IVP-n} with the same initial data $u_0$. Then, 
$$
 u_n
 \in
 C\big((-2S,2S); G^{\sigma,r}(\rr)\big),
 \;\;\; \text{for all } n=1,2,\ldots, 
 $$ 
 and the sequence $\{u_n\}$ converges to $u$ in $C\big([-S,S]; G^{\sigma,r}(\rr)\big)$. 
Furthermore, inequalities \eqref{psiT.u.Bourgain-norm-sobolev} and \eqref{psiT.u.Bourgain-norm-analytic} hold for each $u_n$, uniformly in $n$.
\end{lemma}
We observe that using Duhamel's formula and considering $T=S$, the equality
\begin{equation}
\label{duhamel-n}
\psi_T(t) u_n
=
\psi_T(t) W(t) u_0
-
\frac{1}{p+1} \psi_T(t)\int\limits_0^t W(t-t')\partial_x\left[ (\eta_n \ast \psi_T u_n)^{p+1}  \right]\dd t'
\end{equation}
occurs for all time $t\in \rr$.
Our goal is to show that the sequence $\{\psi_T u_n\}$ is bounded on the Bourgain space $X^{\sigma, s, b}$ for a suitable $\sigma=\sigma(T)>0$.

We omit the proof of the following proposition, since it follows the same steps as the proof of Proposition 2 in \cite{BGK} by using Lemma \ref{psiT.u.Bourgain-norm}.
\begin{proposition}
\label{radius_evolution_seq}
Let $T\ge 1$, $p\ge 2$, $\sigma_0>0$, $s> 3/2$ and $1/2<b<3/4$. Suppose $u$ is a solution of the IVP \eqref{gBenj-IVP} in $C\big([-4T,4T], H^{s+1}(\rr)\big)$ with initial data $u_0\in G^{\sigma_0, s+1}(\rr)$. Then there exist constants $\sigma_1<\sigma_0$ and $K>0$ depending on $s, b, p, \|u_0\|_{G^{\sigma_0, s}}$ and $\alpha_T(u)$ such that the sequence $\{\psi_Tu_n\}$ is bounded in $X^{\sigma(T),s ,b}$ as long as
\begin{equation}\label{sigma.choice}
\sigma(T)
=
\min\big\{ \sigma_1, KT^{-(p^2+3p+2)}\big\}.
\end{equation}
\end{proposition}

Now, we are in position to supply a proof of Theorem \ref{global.sol.gBenj}.

\begin{proof}[Proof of Theorem \ref{global.sol.gBenj}] 
It follows from Proposition \ref{radius_evolution_seq} and Lemma \ref{inc} that the sequence $\{u_n\}$ associated to $u_0$  as in \eqref{gBenj-IVP-n} is bounded in $G^{\tilde{\sigma}(T), s}$ uniformly on $[-T,T]$, where $\tilde{\sigma}(T)= \min\big\{ \sigma_1, KT^{-(p^2+3p+2)}\big\}$.

Proposition \ref{Gds-derivatives} then implies that all the spatial derivatives of $\psi_Tu_n$ are bounded in the strip $S_{\sigma(T)}$, where $\sigma(T)=\tilde{\sigma}(T)/2$.
Since $u_n$ satisfies \eqref{duhamel-n}, the time derivatives of $u_n$ are also bounded in the strip $S_{\sigma(T)}$.
Thus, in particular, for $k=1,2,\ldots$
$$
\{\partial_t u_n\} 
\;\; \text{and}\;\;
\{\partial_x^k u_n\} 
\;\;\text{are equicontinuous families on}\;\;
S_{\sigma(T)}\times (-T, T)
$$
and we can guarantee the existence of a subsequence converging uniformly on compact subsets of $S_{\sigma(T)}\times (-T, T)$ to a smooth function $\tilde{u}$. In the same way, we can choose a subsequence of $\{u_n\}$ such that $\partial_t u_n$, $\partial_x^3 u_n$ and $\mathcal{H} \partial_x^2u_n$ also converge to $\partial_t \tilde{u}$, $\partial_x^3 \tilde{u}$ and $\mathcal{H} \partial_x^2\tilde{u}$, respectively.  We denote this subsequence again by $\{u_n\}$.

Therefore, performing the limit in \eqref{gBenj-IVP-n} and since $\{\eta_n\}$ converges to the Dirac delta function, we get
\begin{equation*}
\left\{
\begin{array}{l}
\partial_t \tilde{u}-l\mathcal{H} \partial_x^2\tilde{u}-\partial_x^3\tilde{u} = -\frac{1}{p+1}\partial_x\left[ (\psi_T \tilde{u})^{p+1}  \right]\\
\tilde{u}(x,0) = u_0(0),
\end{array}
\right.
\end{equation*}
that is, $\tilde{u}$ is a solution to the IVP \eqref{gBenj-IVP} for $(x,t)\in S_{\sigma(T)}\times (-T,T)$.
Moreover, since for every $t\in (-T,T)$
$$
u_n(\cdot, t) \rightarrow \tilde{u}(\cdot, t),
$$
on compact subsets of $S_{\sigma(T)}$ and each $u_n$ is analytic on $S_{\sigma(T)}$, we conclude that $\tilde{u}(\cdot, t)$ is also analytic on $S_{\sigma(T)}$.  
In addition, since the sequence $\{u_n\}$ is bounded in $G^{\sigma(T),s}(\rr)$ uniformly on $[0,T]$, it follows that
$
\tilde{u} \in L^{\infty}\big([0,T]; G^{\sigma(T),s}(\rr)\big).
$
This combined with the local-in-time well-posedness obtained in Section \ref{lwp-section}  yields 
$$
u \in C\big([0,T]; G^{\sigma(T),s}(\rr)\big)
$$
as desired.
\end{proof}


\section{Concluding Remarks} \label{concluding remarks}

Similarly to the study of the analyticity of solution of the gKdV equation posed in the periodic domain one may naturally ask if an analogous study can be done for the generalized Benjamin equation. For motivation, there are a recent works in \cite{HP-12, HHP-11} where the authors considered gKdV equations. Also, there is recent work \cite{ST-17} where the authors considered the quartic gKdV equation and got $ct^{-2}$ as a lower bound  for the radius of analyticity. One may wonder, if such lower bound can be found for the quartic generalized Benjamin equation (for $p = 3$). These questions are being addressed by the authors in the upcoming project.

%
%
%
%
%
%
%
%
%
\vskip 0.3cm
\noindent{\bf Acknowledgements.} 
The first author acknowledges the support from FAPESP (\#2021/04999-9).
The second author acknowledges the grants from FAPESP (\#2020/14833-8) and CNPq (\#307790/2020-7). 
%


\end{document}